\documentclass[11pt]{amsart}

\author[C.~Sanna]{Carlo Sanna$^\dagger$}
\thanks{$\dagger\,$C.~Sanna is a member of GNSAGA of INdAM and of CrypTO, the group of Cryptography and Number~Theory of Politecnico di Torino}
\address{\parbox{\linewidth}{
Politecnico di Torino, Department of Mathematical Sciences\\
Corso Duca degli Abruzzi 24, 10129 Torino, Italy\\[-8pt]}}
\email{carlo.sanna.dev@gmail.com}

\keywords{asymptotic formula; least common multiple; Lucas number}
\subjclass[2010]{Primary: 11B39, Secondary: 11B37, 11N37.}
\title{On the l.c.m.~of shifted Lucas numbers}

\usepackage{amsmath}
\usepackage{amssymb}
\usepackage{amsthm}
\usepackage{booktabs}
\usepackage{cite}
\usepackage{enumitem}
\setenumerate{label=(\roman*), topsep=0.5em, itemsep=0.5em, leftmargin=3.0em}
\usepackage{geometry}
\usepackage{color}
\usepackage{hyperref}
\hypersetup{colorlinks=true}

\newtheorem{theorem}{Theorem}[section]

\newtheorem{lemma}[theorem]{Lemma}
\theoremstyle{remark}

\DeclareMathOperator*{\lcm}{lcm}
\DeclareMathOperator{\Li}{Li}

\allowdisplaybreaks

\begin{document}

\begin{abstract}
Let $(L_n)_{n \geq 1}$ be the sequence of Lucas numbers, defined recursively by $L_1 := 1$, $L_2 := 3$, and $L_{n + 2} := L_{n + 1} + L_n$, for every integer $n \geq 1$.
We determine the asymptotic behavior of $\log \lcm (L_1 + s_1, L_2 + s_2, \dots, L_n + s_n)$ as $n \to +\infty$, for $(s_n)_{n \geq 1}$ a periodic sequence in $\{-1, +1\}$.
We also carry out the same analysis for $(s_n)_{n \geq 1}$ a sequence of independent and uniformly distributed random variables in $\{-1, +1\}$.
These results are Lucas numbers-analogs of previous results obtained by the author for the sequence of Fibonacci numbers.
\end{abstract}

\maketitle

\section{Introduction}

Let $(F_n)_{n \geq 1}$ be the sequence of Fibonacci numbers, defined recursively by $F_1 := 1$, $F_2 := 1$, and $F_{n + 2} := F_{n + 1} + F_n$ for every integer $n \geq 1$.
Guy and Matiyasevich~\cite{MR1712797} proved that 
\begin{equation}\label{equ:lcm-fibs}
\log \lcm (F_1, F_2, \dots, F_n) \sim \frac{3 \log \alpha}{\pi^2} \cdot n^2 \qquad (n \to +\infty),
\end{equation}
where $\lcm$ denotes the least common multiple and $\alpha := \big(1 + \sqrt{5}\big) / 2$ is the golden ratio.
This result was generalized to Lucas sequences, Lehmer sequences, and other sequences with special divisibility properties~\cite{MR1077711, MR1242715, MR1394375, MR3150887, MR1078087, MR1229668, MR993902, MR1114366}.

Motivated by~\eqref{equ:lcm-fibs}, the author considered the least common multiple of \emph{shifted} Fibonacci numbers $F_k \pm 1$, and proved two results~\cite{Sanna21a}.
The first regards periodic sequences of signs:

\begin{theorem}\label{thm:fibs-periodic}
For every periodic sequence $\mathbf{s} = (s_n)_{n \geq 1}$ in $\{-1, +1\}$, there exists an effectively computable rational number $A_\mathbf{s} > 0$ such that
\begin{equation*}
\log \lcm (F_1 + s_1, F_2 + s_2, \dots, F_n + s_n) \sim A_\mathbf{s} \cdot \frac{\log \alpha}{\pi^2} \cdot n^2 \qquad (n \to +\infty) .
\end{equation*}
(Zero terms in the least common multiple are ignored.)
\end{theorem}

By ``effectively computable'' we mean that there exists an algorithm that, given as input the period of the periodic sequence $\mathbf{s}$, returns as output the numerator and denominator of the rational number $A_\mathbf{s}$. 
Indeed, the author computed $A_\mathbf{s}$ for periodic sequences with period length not exceeding $6$~\cite[Tables~1, 2]{Sanna21a}.

The second result regards random sequences of signs. (For similar results on the least common multiple of random sequences, see~\cite{MR4009436, MR3239153, MR4091939, MR4220046, Sanna21b}.)

\begin{theorem}\label{thm:fibs-random}
Let $(s_n)_{n \geq 1}$ be a sequence of independent random variables that are uniformly distributed in $\{-1, +1\}$.
Then
\begin{equation*}
\mathbb{E}\big[\log \lcm (F_1 + s_1, F_2 + s_1, \dots, F_n + s_n)\,\big] \sim \frac{45 \Li_2\!\big(\tfrac1{16}\big)}{2} \cdot \frac{\log \alpha}{\pi^2} \cdot n^2 \qquad (n \to +\infty) ,
\end{equation*}
where $\Li_2(z) := \sum_{n\,=\,1}^\infty z^n / n^2$ denotes the dilogarithm.
\end{theorem}

The purpose of this paper is to establish the analogs of Theorems~\ref{thm:fibs-periodic} and~\ref{thm:fibs-random} for the sequence of Lucas numbers $(L_n)_{n \geq 1}$, defined recursively by $L_1 := 1$, $L_2 := 3$, and $L_{n + 2} := L_{n + 1} + L_n$ for every integer $n \geq 1$.
We remark that the analog of~\eqref{equ:lcm-fibs} is
\begin{equation*}
\log \lcm (L_1, L_2, \dots, L_n) \sim \frac{4 \log \alpha}{\pi^2} \cdot n^2 \qquad (n \to +\infty) ,
\end{equation*}
which follows from a result of B\'ezivin~\cite{MR1078087}.

Our first result is the following analog of Theorem~\ref{thm:fibs-periodic}.

\begin{theorem}\label{thm:periodic}
For every periodic sequence $\mathbf{s} = (s_n)_{n \geq 1}$ in $\{-1, +1\}$, there exists an effectively computable rational number $B_\mathbf{s} > 0$ such that
\begin{equation*}
\log \lcm (L_1 + s_1, L_2 + s_2, \dots, L_n + s_n) \sim B_\mathbf{s} \cdot \frac{\log \alpha}{\pi^2} \cdot n^2 \qquad (n \to +\infty) .
\end{equation*}
\end{theorem}

We computed $B_\mathbf{s}$ for periodic sequences with period length at most $6$, see Table~\ref{tab:Bs}.
We notice that for such sequences $A_\mathbf{s}$ takes $8$ different values, while $B_\mathbf{s}$ takes $58$.

\begin{table}[ht]
\label{tab:Bs}
\caption{Values of $B_\mathbf{s}$ for periodic sequences $\mathbf{s}$ with period length at most $6$.}
\centering
\begin{tabular}{cc|cc|cc|cc}
  \toprule
  $\mathbf{s}$ & $D_{\mathbf{s}}$ & $\mathbf{s}$ & $D_{\mathbf{s}}$ & $\mathbf{s}$ & $D_{\mathbf{s}}$ & $\mathbf{s}$ & $D_{\mathbf{s}}$ \\\hline
  \texttt{-} & $7/2$ & \texttt{--++-} & $128399/46080$ & \texttt{----++} & $601/192$ & \texttt{+---+-} & $581/192$ \rule{0pt}{2.6ex}\\
  \texttt{+} & $45/16$ & \texttt{--+++} & $24931/9216$ & \texttt{---+--} & $103/32$ & \texttt{+---++} & $179/64$ \\
  \texttt{-+} & $45/16$ & \texttt{-+---} & $1801/576$ & \texttt{---+-+} & $287/96$ & \texttt{+--+-+} & $65/24$ \\
  \texttt{+-} & $7/2$ & \texttt{-+--+} & $3095/1024$ & \texttt{---++-} & $73/24$ & \texttt{+--++-} & $527/192$ \\
  \texttt{--+} & $179/64$ & \texttt{-+-+-} & $1163/384$ & \texttt{---+++} & $45/16$ & \texttt{+--+++} & $161/64$ \\
  \texttt{-+-} & $73/24$ & \texttt{-+-++} & $8917/3072$ & \texttt{--+---} & $581/192$ & \texttt{+-+---} & $99/32$ \\
  \texttt{-++} & $45/16$ & \texttt{-++--} & $62909/23040$ & \texttt{--+-+-} & $215/64$ & \texttt{+-+--+} & $93/32$ \\
  \texttt{+--} & $139/48$ & \texttt{-++-+} & $2195/768$ & \texttt{--+-++} & $601/192$ & \texttt{+-+-++} & $157/48$ \\
  \texttt{+-+} & $65/24$ & \texttt{-+++-} & $12331/4608$ & \texttt{--++--} & $527/192$ & \texttt{+-++--} & $139/48$ \\
  \texttt{++-} & $493/192$ & \texttt{-++++} & $133/48$ & \texttt{--++-+} & $161/64$ & \texttt{+-+++-} & $103/32$ \\
  \texttt{---+} & $45/16$ & \texttt{+----} & $2399/768$ & \texttt{--+++-} & $73/24$ & \texttt{+-++++} & $287/96$ \\
  \texttt{--+-} & $91/32$ & \texttt{+---+} & $4339/1536$ & \texttt{--++++} & $45/16$ & \texttt{++----} & $139/48$ \\
  \texttt{--++} & $9/4$ & \texttt{+--+-} & $4531/1536$ & \texttt{-+----} & $103/32$ & \texttt{++---+} & $65/24$ \\
  \texttt{-+--} & $7/2$ & \texttt{+--++} & $60269/23040$ & \texttt{-+---+} & $287/96$ & \texttt{++--+-} & $527/192$ \\
  \texttt{-++-} & $47/16$ & \texttt{+-+--} & $763/256$ & \texttt{-+--++} & $45/16$ & \texttt{++--++} & $161/64$ \\
  \texttt{-+++} & $9/4$ & \texttt{+-+-+} & $739/256$ & \texttt{-+-+--} & $73/24$ & \texttt{++-+--} & $87/32$ \\
  \texttt{+---} & $25/8$ & \texttt{+-++-} & $409/144$ & \texttt{-+-++-} & $557/192$ & \texttt{++-+-+} & $81/32$ \\
  \texttt{+--+} & $39/16$ & \texttt{+-+++} & $1055/384$ & \texttt{-+-+++} & $171/64$ & \texttt{++-+++} & $449/192$ \\
  \texttt{+-++} & $45/16$ & \texttt{++---} & $1603/576$ & \texttt{-++---} & $527/192$ & \texttt{+++---} & $139/48$ \\
  \texttt{++--} & $3$ & \texttt{++--+} & $1549/576$ & \texttt{-++--+} & $161/64$ & \texttt{+++--+} & $65/24$ \\
  \texttt{++-+} & $39/16$ & \texttt{++-+-} & $4361/1536$ & \texttt{-++-+-} & $73/24$ & \texttt{+++-+-} & $103/32$ \\
  \texttt{+++-} & $7/2$ & \texttt{++-++} & $12455/4608$ & \texttt{-+++--} & $493/192$ & \texttt{+++-++} & $287/96$ \\
  \texttt{----+} & $4087/1280$ & \texttt{+++--} & $4043/1536$ & \texttt{-+++-+} & $449/192$ & \texttt{++++--} & $87/32$ \\
  \texttt{---+-} & $9709/3072$ & \texttt{+++-+} & $2107/768$ & \texttt{-++++-} & $557/192$ & \texttt{++++-+} & $81/32$ \\
  \texttt{---++} & $130987/46080$ & \texttt{++++-} & $2113/768$ & \texttt{-+++++} & $171/64$ & \texttt{+++++-} & $73/24$ \\
  \texttt{--+--} & $5981/1920$ & \texttt{-----+} & $157/48$ & \texttt{+-----} & $99/32$ & \texttt{ } & $ $ \\
  \texttt{--+-+} & $1735/576$ & \texttt{----+-} & $215/64$ & \texttt{+----+} & $93/32$ & \texttt{ } & $ $ \\

\bottomrule
\end{tabular}
\end{table}

Our second result is the following analog of Theorem~\ref{thm:fibs-random}.

\begin{theorem}\label{thm:random}
Let $(s_n)_{n \geq 1}$ be a sequence of independent random variables that are uniformly distributed in $\{-1, +1\}$.
Then
\begin{equation*}
\mathbb{E}\big[\log \lcm (L_1 + s_1, L_2 + s_1, \dots, L_n + s_n) \,\big] \sim C \cdot \frac{\log \alpha}{\pi^2} \cdot n^2 \qquad (n \to +\infty) ,
\end{equation*}
where 
\begin{equation*}
C := \tfrac{243}{128} + \tfrac{27}{8}\Li_2\!\big(\tfrac1{4}\big) + \tfrac{9}{8}\Li_2\!\big(\tfrac1{16}\big) + \tfrac{3}{16}\Li_2\!\big(\tfrac1{16}; \tfrac1{3}\big) + \tfrac{3}{32}\Li_2\!\big(\tfrac1{16}; \tfrac{2}{3}\big)
\end{equation*}
and $\Li_2(z; a) := \sum_{n\,=\,1}^\infty z^n / (n + a)^2$.
\end{theorem}

The proofs of Theorems~\ref{thm:periodic} and~\ref{thm:random} employ methods similar to those used in the proofs of Theorems~\ref{thm:fibs-periodic} and~\ref{thm:fibs-random}.
However, the details are more involved because the multiplicative expressions of shifted Lucas numbers in terms of Fibonacci and Lucas numbers (see Lemma~\ref{lem:shifted} below) are more complex than those of shifted Fibonacci numbers~(see \cite[Lemma~2.3]{Sanna21a}).

\section{Notation}

We employ the Landau--Bachmann ``Big Oh'' notation $O$ with its usual meaning. 
Any dependence of the implied constants is indicated with subscripts.
We let $\lfloor x \rfloor$ denote the greatest integer not exceeding $x$.
We reserve the letter $p$ for prime numbers, and we write $\nu_p(n)$, $\varphi(n)$, and $\mu(n)$, for the $p$-adic valuation, the Euler function, and the M\"obius function of a positive integer $n$, respectively.

\section{Preliminaries on Fibonacci and Lucas numbers}

It is well known that the Binet formulas
\begin{equation}\label{equ:binet}
F_n = \frac{\alpha^n - \beta^n}{\alpha - \beta} \quad \text{ and } \quad L_n = \alpha^n + \beta^n ,
\end{equation}
where $\alpha := \big(1 + \sqrt{5}) / 2$ and $\beta := \big(1 - \sqrt{5}) / 2$, hold for every integer $n \geq 1$.
Let $\Phi_1 := 1$ and
\begin{equation}\label{equ:Phin}
\Phi_n := \prod_{\substack{1 \,\leq\, k \,\leq\, n \\ (n,\, k) \,=\, 1}} \left(\alpha - \mathrm{e}^{\frac{2\pi\mathbf{i}k}{n}} \beta \right)
\end{equation}
for every integer $n \geq 2$.
It can be proved that each $\Phi_n$ is a positive integer (see~\cite[p.~428]{MR491445} for $\Phi_n \in \mathbb{Z}$ and~\cite[p.~979]{MR4003803} for $\Phi_n > 0$).
For every prime number $p$, let $z(p)$ be the minimum integer $n \geq 1$ such that $p \mid F_n$.
It is well known that $z(p)$ exists.

\begin{lemma}\label{lem:Phizppv}
Let $n \geq 1$ be an integer and suppose that $p$ is a prime factor of $\Phi_n$.
Then $n = z(p) p^v$ for some integer $v \geq 0$.
Furthermore, if $v \geq 1$ and $(p, n) \neq (2, 6)$ then $p \mid\mid \Phi_n$.
(Note that $2^2 \mid\mid \Phi_6 = 4$.)
\end{lemma}
\begin{proof}
See~\cite[Discussion before Lemma~6]{MR491445}.
\end{proof}

\begin{lemma}\label{lem:Phiprodlcm}
Let $x > 1$ and $\mathcal{S} \subseteq \mathbb{N} \cap [1, x]$.
Then we have 
\begin{equation}\label{equ:Phi-prod-lcm}
\prod_{d \,\in\, \mathcal{S}} \Phi_d = R_\mathcal{S} \cdot \lcm_{d \,\in\, \mathcal{S}} \Phi_d ,
\end{equation}
where $R_\mathcal{S}$ is a positive integer such that $p \leq x$ and $\nu_p(R_\mathcal{S}) \leq 2\log x / \log p$ for every prime number $p$ dividing $R_\mathcal{S}$.
In particular, these conditions on $R_\mathcal{S}$ imply that $\log R_\mathcal{S} = O(x)$.
\end{lemma}
\begin{proof}
Let $R_\mathcal{S}$ be defined by~\eqref{equ:Phi-prod-lcm}.
Clearly, $R_\mathcal{S}$ is a positive integer.
Let $p$ be a prime factor of $R_\mathcal{S}$ and let $d_1, \dots, d_m$ be all the $m$ pairwise distinct elements of $\{d \in \mathcal{S} : p \mid \Phi_d\}$.
Without loss of generality, we can assume that $\nu_p(\Phi_{d_1}) \leq \cdots \leq \nu_p(\Phi_{d_m})$.
In light of Lemma~\ref{lem:Phizppv}, we have that $m \leq \lfloor \log x / \log p \rfloor + 1$ and $\nu_p(\Phi_{d_i}) \leq 2$ for every integer $i \in {[1, m)}$.
Therefore,
\begin{equation*}
1 \leq \nu_p(R_\mathcal{S}) = \sum_{i \,\leq\, m} \nu_p(\Phi_{d_i}) - \max_{i \,\leq\, m} \nu_p(\Phi_{d_i}) = \sum_{i \,<\, m } \nu_p(\Phi_{d_i}) \leq 2(m - 1) \leq \frac{2 \log x}{\log p} .
\end{equation*}
In particular, it must be $m \geq 2$, so that Lemma~\ref{lem:Phizppv} yields $p \leq x$.
Finally, 
\begin{equation*}
\log R_\mathcal{S} \leq \log \prod_{p \,\leq\, x} p^{2 \log x / \log p} = \#\{p : p \leq x\} \cdot 2 \log x = O(x) ,
\end{equation*}
since the number of primes not exceeding $x$ is $O(x / \log x)$.
\end{proof}

\begin{lemma}\label{lem:logPhin}
We have $\log \Phi_n = \varphi(n) \log \alpha + O(1)$, for all integers $n \geq 1$.
\end{lemma}
\begin{proof}
See~\cite[Lemma~2.1(iii)]{MR4003803}.
\end{proof}

For all integers $a, n \geq 1$, let us define
\begin{equation}\label{equ:Dan}
\mathcal{D}_a(n) := \{d \in \mathbb{N} : d \mid an, \, (an / d, a) = 1\} .
\end{equation}
Note that $\mathcal{D}_1(n)$ is the set of positive divisors of $n$.
Moreover, we have the following:

\begin{lemma}\label{lem:Dapn}
We have $\mathcal{D}_{ap}(n) = \mathcal{D}_a(pn) \setminus \mathcal{D}_a(n)$, for all integers $a, n \geq 1$ and for every prime number $p$ not dividing $a$.
\end{lemma}
\begin{proof}
On the one hand, if $d \in \mathcal{D}_{ap}(n)$ then $d \mid apn$ and $(apn / d, ap) = 1$.
Hence, $(apn / d, a) = 1$, which implies $d \in \mathcal{D}_a(pn)$, and $d \nmid an$, which implies $d \notin \mathcal{D}_a(n)$.
Thus $d \in \mathcal{D}_a(pn) \setminus \mathcal{D}_a(n)$.

On the other hand, if $d \in \mathcal{D}_a(pn) \setminus \mathcal{D}_a(n)$ then $d \mid apn$, $(apn / d, a) = 1$, and $d \nmid an$.
Hence, $(apn / d, ap) = 1$ and consequently $d \in \mathcal{D}_{ap}(n)$.
\end{proof}

\begin{lemma}\label{lem:products}
We have
\begin{equation*}
F_n = \prod_{d \,\in\, \mathcal{D}_1(n)} \Phi_d, \quad L_n = \prod_{d \,\in\, \mathcal{D}_2(n)} \Phi_d, \quad \frac{F_{3n}}{F_n} = \prod_{d \,\in\, \mathcal{D}_3(n)} \Phi_d, \quad \frac{L_{3n}}{L_n} = \prod_{d \,\in\, \mathcal{D}_6(n)} \Phi_d ,
\end{equation*}
for all integers $n \geq 1$.
\end{lemma}
\begin{proof}
The four identities follow from~\eqref{equ:binet}, \eqref{equ:Phin}, and the fact that $\mathcal{D}_2(n) = \mathcal{D}_1(2n) / \mathcal{D}_1(n)$, $\mathcal{D}_3(n) = \mathcal{D}_1(3n) / \mathcal{D}_1(n)$, and $\mathcal{D}_6(n) = \mathcal{D}_2(3n) / \mathcal{D}_2(n)$, as consequence of Lemma~\ref{lem:Dapn}.
\end{proof}

The next lemma belong to the folkore and makes possible to write shifted Lucas numbers as products or ratios of Fibonacci and Lucas numbers.

\begin{lemma}\label{lem:shifted}
We have
\begin{alignat*}{4}
L_{4k} - 1 &= L_{6k} / L_{2k}, \qquad & L_{4k} + 1 &= F_{6k} / F_{2k} , \\
L_{4k + 1} - 1 &= 5 F_{2k} F_{2k + 1}, \qquad & L_{4k + 1} + 1 &= L_{2k} L_{2k + 1} , \\
L_{4k + 2} - 1 &= F_{6k + 3} / F_{2k + 1}, \qquad & L_{4k + 2} + 1 &= L_{6k + 3} / L_{2k + 1} , \\
L_{4k + 3} - 1 &= L_{2k + 1} L_{2k + 2}, \qquad & L_{4k + 3} + 1 &= 5 F_{2k + 1} F_{2k + 2} ,
\end{alignat*}
for all integers $k \geq 1$.
\end{lemma}
\begin{proof}
Taking into account that $\alpha + \beta = 1$, $\alpha - \beta = \sqrt{5}$, and $\alpha \beta = -1$, the eight identities follow by substituting $(X, Y) = \left(\alpha^{2k}, \pm \beta^{2k}\right)$ and $(X, Y) = \left(\alpha^{2k + 1}, \pm \beta^{2k + 1}\right)$ into
\begin{align*}
X^2 + Y^2 + XY &= (X^3 - Y^3) / (X - Y) , \\
 \alpha X^2 + \beta Y^2 + XY &= (X + Y)(\alpha X + \beta Y) ,
\end{align*}
and using the Binet formulas~\eqref{equ:binet}.
\end{proof}

Now define the sets $(\mathcal{E}_{\pm}(n))_{n \geq 4}$ and the integers $(e_\pm(n))_{n \geq 4}$ by
\begin{alignat*}{4}
\mathcal{E}_{-}(4k)   &= \mathcal{D}_6(2k) ,                                 \quad & \mathcal{E}_+(4k)     &= \mathcal{D}_3(2k) , \\
\mathcal{E}_{-}(4k+1) &= \mathcal{D}_1(2k) \cup \mathcal{D}_1(2k + 1) ,      \quad & \mathcal{E}_+(4k + 1) &= \mathcal{D}_2(2k) \cup \mathcal{D}_2(2k + 1) , \\
\mathcal{E}_{-}(4k+2) &= \mathcal{D}_3(2k + 1) ,                             \quad & \mathcal{E}_+(4k + 2) &= \mathcal{D}_6(2k + 1) , \\
\mathcal{E}_{-}(4k+3) &= \mathcal{D}_2(2k + 1) \cup \mathcal{D}_2(2k + 2) ,  \quad & \mathcal{E}_+(4k + 3) &= \mathcal{D}_1(2k + 1) \cup \mathcal{D}_1(2k + 2) ,
\end{alignat*}
and $e_-(4k + 1) = e_+(4k + 3) = 5$, for every integer $k \geq 1$; while $e_\pm(n) = 1$ for all the other integers $n \geq 4$.

\begin{lemma}\label{lem:shifted-product}
We have
\begin{equation*}
L_n + s = e_s(n) \prod_{d \,\in\, \mathcal{E}_s(n)} \Phi_d ,
\end{equation*}
for all integers $n \geq 4$ and for all $s \in \{-1, +1\}$.
\end{lemma}
\begin{proof}
Noticing that $\mathcal{D}_1(m) \cap \mathcal{D}_1(m + 1) = \{1\}$ and $\mathcal{D}_2(m) \cap \mathcal{D}_2(m + 1) = \varnothing$, for every integer $m \geq 1$, the claim follows from Lemma~\ref{lem:products} and Lemma~\ref{lem:shifted}.
\end{proof}

\section{Further preliminaries}

For every sequence $\mathbf{s} = (s_k)_{k \geq 1}$ in $\{-1, +1\}$ and for all integer $n \geq 4$, let us define
\begin{equation*}
\ell_\mathbf{s}(n) := \lcm_{4 \,\leq\, k \,\leq\, n} (L_k + s_k) \quad \text{ and } \quad \mathcal{L}_\mathbf{s}(n) := \bigcup_{a \,\in\, \{1,2,3,6\}} \; \bigcup_{h \,\in\, \mathcal{K}_{a,\mathbf{s}}(n)} \mathcal{D}_a(h) ,
\end{equation*}
where
\begin{align*}
\mathcal{K}_{1, \mathbf{s}}(n) &:= \big\{h \leq n / 2 : s_{2h - 1} = (-1)^h \,\lor\, s_{2h + 1} = (-1)^{h + 1} \big\} , \\
\mathcal{K}_{2, \mathbf{s}}(n) &:= \big\{h \leq n / 2 : s_{2h - 1} = (-1)^{h + 1} \,\lor\, s_{2h + 1} = (-1)^{h} \big\} , \\
\mathcal{K}_{3, \mathbf{s}}(n) &:= \big\{h \leq n / 2 : s_{2h} = (-1)^h \big\} , \\
\mathcal{K}_{6, \mathbf{s}}(n) &:= \big\{h \leq n / 2 : s_{2h} = (-1)^{h + 1} \big\} .
\end{align*}
The next lemma will be fundamental in the proofs of Theorem~\ref{thm:periodic} and Theorem~\ref{thm:random}.

\begin{lemma}\label{lem:logellsn_asymp}
We have
\begin{equation}\label{equ:logellsn}
\log \ell_\mathbf{s}(n) = \sum_{d \,\in\, \mathcal{L}_\mathbf{s}(n)} \varphi(d) \log \alpha + O(n) ,
\end{equation}
for all integers $n \geq 4$ and for all sequences $\mathbf{s} = (s_k)_{k \geq 1}$ in $\{-1, +1\}$.
\end{lemma}
\begin{proof}
From the definition of $\mathcal{E}_\pm(k)$, it follows that the symmetric difference of $\mathcal{L}_\mathbf{s}(n)$ and
\begin{equation*}
\mathcal{L}_\mathbf{s}^\prime(n) := \bigcup_{4 \,\leq\, k \,\leq\, n} \mathcal{E}_{s_k}(k)
\end{equation*}
is contained in $[1,4] \cup \big[(n - 1)/2, (n + 1)/2\big]$.
Therefore, it suffices to prove~\eqref{equ:logellsn} with $\mathcal{L}_\mathbf{s}^\prime(n)$ in place of $\mathcal{L}_\mathbf{s}(n)$.

For every integer $m \geq 1$, write $m = m^{(\leq)} \cdot m^{(>)}$, where $m^{(\leq)}$, respectively $m^{(>)}$, is a positive integer having all prime factors not exceeding $3n$, respectively greater than $3n$.
Note that for every integer $k \in [4, n]$ we have $\mathcal{E}_\pm(k) \subseteq [1, 3n]$, and consequently $\mathcal{L}_{\mathbf{s}}^\prime(n) \subseteq [1, 3n]$.

On the one hand, from Lemma~\ref{lem:shifted-product} and Lemma~\ref{lem:Phiprodlcm} (with $x = 3n$), we have that
\begin{align}\label{equ:logellsn_1}
\ell_\mathbf{s}^{(>)}(n) &= \lcm_{\substack{4 \,\leq\, k \,\leq\, n}} (L_k + s_k)^{(>)} = \lcm_{\substack{4 \,\leq\, k \,\leq\, n}} \prod_{d \,\in\, \mathcal{E}_{s_k}\!(k)} \Phi_d^{(>)} \\
 &= \lcm_{\substack{4 \,\leq\, k \,\leq\, n}} \; \lcm_{d \,\in\, \mathcal{E}_{s_k}\!(k)} \Phi_d^{(>)} = \lcm_{d \,\in\, \mathcal{L}_\mathbf{s}^\prime(n)} \Phi_d^{(>)} = \prod_{d \,\in\, \mathcal{L}_\mathbf{s}^\prime(n)} \Phi_d^{(>)} , \nonumber 
\end{align}
and
\begin{equation}\label{equ:logellsn_2}
\prod_{d \,\in\, \mathcal{L}_\mathbf{s}^\prime(n)} \Phi_d^{(\leq)} = \mathrm{e}^{O(n)} .
\end{equation}

On the other hand, again from Lemma~\ref{lem:shifted-product} and Lemma~\ref{lem:Phiprodlcm}, we get that
\begin{equation}\label{equ:logellsn_3}
\ell_\mathbf{s}^{(\leq)}(n) = \lcm_{\substack{4 \,\leq\, k \,\leq\, n}} (L_k + s_k)^{(\leq)} = \lcm_{\substack{4 \,\leq\, k \,\leq\, n}} \prod_{d \,\in\, \mathcal{E}_{s_k}\!(k)} \Phi_d^{(\leq)} = \lcm_{\substack{4 \,\leq\, k \,\leq\, n}} R_k = \mathrm{e}^{O(n)} ,
\end{equation}
where each $R_k$ is a positive integer such that $p \leq 3n$ and $\nu_p(R_k) \leq 2\log(3n) / \log p$ for every prime number $p$ dividing $R_k$.

Therefore, from~\eqref{equ:logellsn_1}, \eqref{equ:logellsn_2}, and~\eqref{equ:logellsn_3}, we find that
\begin{equation*}
\log \ell_\mathbf{s}(n) = \log\!\bigg(\prod_{d \,\in\, \mathcal{L}_\mathbf{s}^\prime(n)} \Phi_d\bigg) + O(n) = \sum_{d \,\in\, \mathcal{L}_\mathbf{s}^\prime(n)} \varphi(d)\log \alpha + O\!\left(\#\mathcal{L}_\mathbf{s}^\prime(n)\right) + O(n) ,
\end{equation*}
where we used Lemma~\ref{lem:logPhin}.
Since $\#\mathcal{L}_\mathbf{s}^\prime(n) \leq 3n$, the claim follows.
\end{proof}

For all integers $r, m \geq 1$ and for every $x > 1$, let us define the arithmetic progression
\begin{equation*}
\mathcal{A}_{r,m}(x) := \big\{n \leq x : n \equiv r \!\!\!\pmod m\big\} ,
\end{equation*}
and put
\begin{equation*}
c_{r, m} := \frac1{m} \prod_{\substack{p \,\mid\, m \\ \!\!p \,\mid\, r}} \left(1 + \frac1{p}\right)^{-1} \prod_{\substack{p \,\mid\, m \\ \!\! p \,\nmid\, r}} \left(1 - \frac1{p^2}\right)^{-1} .
\end{equation*}
Also, for every integer $q \geq 1$ and for all $z, w \in [0, 1]$, let us define
\begin{equation*}
T_q(z, w) := \begin{cases}
1 & \text{ if $z = 0$;}  \\
\frac{(1 - z)\Li_2(z)}{z} & \text{ if $z > 0$ and $w = 1$;} \\
\frac{(1 - zw)\Li_2(z^q w)}{q^2zw} + \frac{1 - z}{z} \sum_{j \,=\, 1}^{q - 1} z^j \bigg(\frac1{j^2} + \frac{\Li_2(z^q w; j/q)}{q^2} \bigg) & \text{ if $z > 0$ and $w < 1$.}
\end{cases}
\end{equation*}
We need some asymptotic formulas for sums of the Euler totient function over $\mathcal{A}_{r, m}(x)$.

\begin{lemma}\label{lem:phisum}
Let $r,m,q \geq 1$ be integers, and let $z, w \in {[0,1]}$.
Then we have
\begin{equation*}
\sum_{n \,\in\, \mathcal{A}_{r, m}(x)} \varphi(n) \left(1 - z^{\left\lfloor \tfrac{x}{n} \right\rfloor} w^{\left\lfloor \tfrac{x}{qn} \right\rfloor} \right) = \frac{3}{\pi^2} c_{r,m} T_q(z, w) x^2 + O_{r,m}\!\left(x (\log x)^2\right) ,
\end{equation*}
for all $x \geq 2$, where for $z = 0$ the error term can be improved to $O_{r,m}(x \log x)$.
\end{lemma}
\begin{proof}
For $z = 0$, or $z > 0$ and $w = 1$, see~\cite[Lemma~3.4, Lemma~3.5]{Sanna21a}.
Suppose that $z > 0$ and $w < 1$, and let
\begin{equation*}
S_{r,m}(x) := \sum_{n \,\in\, \mathcal{A}_{r, m}(x)} \varphi(n) = \frac{3}{\pi^2} c_{r,m} x^2 + O_{r,m}(x \log x) ,
\end{equation*}
for all $x \geq 2$.
For every integer $k \geq 1$, we have that $\lfloor x / n \rfloor = k$ if and only if $x / (k + 1) < n \leq x / k$, and in such a case it holds $\lfloor x / (qn) \rfloor = \lfloor k / q \rfloor$.
Therefore, we have
\begin{align*}
\sum_{n \,\in\, \mathcal{A}_{r, m}(x)} & \varphi(n) \left(1 - z^{\left\lfloor \tfrac{x}{n} \right\rfloor} w^{\left\lfloor \tfrac{x}{qn} \right\rfloor} \right) = \sum_{k \,\leq\, x} \left(1 - z^k w^{\left\lfloor \tfrac{k}{q} \right\rfloor}\right) \! \left(S_{r,m}\!\left(\frac{x}{k}\right) - S_{r,m}\!\left(\frac{x}{k + 1}\right)\right) \\
 &= \sum_{k \,\leq\, x} \left(\left(1 - z^k w^{\left\lfloor \tfrac{k}{q} \right\rfloor}\right) - \left(1 - z^{k-1} w^{\left\lfloor \tfrac{k-1}{q} \right\rfloor}\right)\right) S_{r,m}\!\left(\frac{x}{k}\right) \\
 &= \sum_{k \,\leq\, x} z^{k - 1} \! \left(w^{\left\lfloor \tfrac{k - 1}{q} \right\rfloor} - z w^{\left\lfloor \tfrac{k}{q} \right\rfloor}\right) \! \left(\frac{3}{\pi^2} \cdot \frac{c_{r,m} x^2}{k^2} + O_{r,m}\!\left(\frac{x \log x}{k} \right)\right) \\
 &= \frac{3 c_{r,m}}{\pi^2}U_q(z, w) x^2 + O_{r,m}\bigg(\sum_{k \,>\, x} \frac{x^2}{k^2} \bigg) + O_{r,m}\bigg(\sum_{k \,\leq\, x} \frac{x \log x}{k}\bigg) \\
 &= \frac{3 c_{r,m}}{\pi^2}U_q(z, w) x^2 + O_{r,m}\!\left(x (\log x)^2\right) ,
\end{align*}
where
\begin{align*}
U_q(z, w) &:= \sum_{k \,=\, 1}^\infty \frac{z^{k - 1}}{k^2} \! \left(w^{\left\lfloor \tfrac{k - 1}{q} \right\rfloor} - z w^{\left\lfloor \tfrac{k}{q} \right\rfloor}\right) \\
 &= (1 - z) \sum_{k \,=\, 1}^{q - 1} \frac{z^{k - 1}}{k^2} + \sum_{j \,=\, 0}^{q - 1} \sum_{h \,=\, 1}^\infty \frac{z^{qh + j - 1} w^h}{(qh + j)^2} \begin{cases} (w^{-1} - z) & \text{ if } j = 0; \\ (1 - z) & \text{ if } j \geq 1; \end{cases} \\
 &= \frac{(1 - zw)\Li_2(z^q w)}{q^2zw} + \frac{1 - z}{z} \sum_{j \,=\, 1}^{q - 1} z^j \! \left( \frac1{j^2} + \frac{\Li_2(z^q w; j / q)}{q^2} \right) ,
\end{align*}
as desired.
\end{proof}

\section{Proof of Theorem~\ref{thm:periodic}}

For all integers $a, r, m \geq 1$, let $\mathcal{T}_{a,r,m}$ be the set of $t \in \{1, \dots, m\}$ such that there exists an integer $u \geq 1$ satisfying $tu \equiv r \pmod m$ and $(a, u) = 1$.
Moreover, for each $t \in \mathcal{T}_{a,r,m}$ let $u_{a,r,m}(t)$ be the minimum of the possible integers $u$.

\begin{lemma}\label{lem:unionDan}
Let $a, r, m \geq 1$ be integers.
Then we have
\begin{equation*}
\bigcup_{n \,\in\, \mathcal{A}_{r,m}(x)} \mathcal{D}_a(n) = \bigcup_{t \,\in\, \mathcal{T}_{a,r,m}} \mathcal{A}_{at, am}\!\left(\frac{ax}{u_{a,r,m}(t)}\right) ,
\end{equation*}
for every $x > 1$.
\end{lemma}
\begin{proof}
On the one hand, suppose that $d \in \mathcal{D}_a(n)$ for some $n \in \mathcal{A}_{r,m}(x)$.
In~particular, $d \mid an$ and $a \mid d$.
Put $u := an / d$ and let $t \in \{1, \dots, m\}$ be such that $d / a \equiv t \pmod m$.
Then $tu \equiv (d/a)(an/d) \equiv n \equiv r \pmod m$ and $(a, u) = 1$, so that $t \in \mathcal{T}_{a,r,m}$.
Moreover, $d \equiv at \pmod {am}$ and $d = an / u \leq ax / u_{a,r,m}(t)$, so that $d \in \mathcal{A}_{at,am}\big(ax / u_{a,r,m}(t))$.

On the other hand, suppose that $d \in \mathcal{A}_{at,am}\big(ax / u_{a,r,m}(t))$ for some $t \in \mathcal{T}_{a,r,m}$.
Since $d \equiv at \pmod {am}$, we have that $a \mid d$ and $d / a \equiv t \pmod m$.
Let $u := u_{a,r,m}(t)$ and $n := (d / a)u$.
Then $d \mid an$ and $(an / d, a) = (u, a) = 1$, so that $d \in \mathcal{D}_a(n)$.
Moreover, $n \equiv (d/a) u \equiv t u \equiv r \pmod m$ and $n \leq ((ax / u) / a) u = x$, so that $n \in \mathcal{A}_{r,m}(x)$.
\end{proof}

Let $\mathbf{s} = (s_n)_{n \geq 1}$ be a periodic sequence in $\{-1, +1\}$, and let $m$ be the length of its period.
Then, for each $a \in \{1,2,3,6\}$, there exists $\mathcal{R}_{a, \mathbf{s}} \subseteq \{1, \dots, 2m\}$ such that
\begin{equation*}
\mathcal{K}_{a,\mathbf{s}}(n) = \bigcup_{r \,\in\, \mathcal{R}_{a, \mathbf{s}}} \mathcal{A}_{r, 2m}(n / 2) ,
\end{equation*}
for all integers $n \geq 4$.
From Lemma~\ref{lem:unionDan} and the fact that arithmetic progressions modulo $m$, $2m$, and $3m$ can be written as unions of arithmetic progressions modulo $6m$, it follows that there exist $\mathcal{R}_\mathbf{s} \subseteq \{1, \dots, 6m\}$ and positive rational numbers $(q_r)_{r \in \mathcal{R}_\mathbf{s}}$ such that
\begin{equation*}
\mathcal{L}_\mathbf{s}(n) = \bigcup_{a \,\in\, \{1,2,3,6\}} \; \bigcup_{h \,\in\, \mathcal{K}_{a,\mathbf{s}}(n)} \mathcal{D}_a(h) = \bigcup_{r \,\in\, \mathcal{R}_\mathbf{s}} \mathcal{A}_{r, 6m}(q_r n) .
\end{equation*}
Now Lemma~\ref{lem:logellsn_asymp} and Lemma~\ref{lem:phisum} yield that
\begin{equation*}
\log \ell_\mathbf{s}(n) = \sum_{r \,\in\, \mathcal{R}_\mathbf{s}} \; \sum_{d \,\in\, \mathcal{A}_{r, 6m}(q_r n)} \varphi(d) \log \alpha + O(n) = B_\mathbf{s} \cdot \frac{\log \alpha}{\pi^2} \cdot n^2 + O_\mathbf{s}(n \log n) ,
\end{equation*}
for all integers $n \geq 4$, where
\begin{equation*}
B_\mathbf{s} := 3 \sum_{r \,\in\, \mathcal{R}_\mathbf{s}} c_{r, 6m} q_r^2
\end{equation*}
is a positive rational number, which is effectively computable in terms of $s_1, \dots, s_m$.

The proof is complete.

\section{Proof of Theorem~\ref{thm:random}}

For all integers $a, d \geq 1$ and for every $x > 1$, let us define
\begin{equation*}
\mathcal{H}_a(d) := \big\{h \in \mathbb{N} : d \in \mathcal{D}_a(h)\big\} \quad\text{ and }\quad \mathcal{H}_a(d;x) := \mathcal{H}_a(d) \cap [1, x] .
\end{equation*}
We need the following easy result.

\begin{lemma}\label{lem:Had}
For all integers $a,b,d \geq 1$ and for every $x > 1$, we have that:
\begin{enumerate}
\item If $a \mid d$ then $\mathcal{H}_a(d) = \big\{\tfrac{d}{a} v : v \in \mathbb{N},\; (a, v) = 1\big\}$, otherwise $\mathcal{H}_a(d) = \varnothing$.
\item If $a \mid d$ then $\#\mathcal{H}_a(d; x) = \sum_{b \,\mid\, a} \mu(b) \!\left\lfloor \tfrac{ax}{bd} \right\rfloor$, otherwise $\#\mathcal{H}_a(d; x) = 0$.
\item If $a \neq b$ then $\mathcal{H}_a(d) \cap \mathcal{H}_b(d) = \varnothing$.
\end{enumerate}
\end{lemma}
\begin{proof}
The claim (i) follows easily from~\eqref{equ:Dan}, while (ii) is a consequence of (i) and the inclusion-exclusion principle.
Regarding (iii), suppose that $h \in \mathcal{H}_a(d) \cap \mathcal{H}_b(d)$.
Then it follows from (i) that $h = dv / a = dw / b$, for some integers $v,w \geq 1$ such that $(a, v) = 1$ and $(b, w) = 1$.
Hence, $bv = aw$ and by the conditions of coprimality it follows that $a = b$.
\end{proof}

In what follows, let $\mathbf{s} = (s_n)_{n \geq 1}$ be a sequence of independent and uniformly distributed random variables in $\{-1, +1\}$.
Furthermore, define
\begin{equation}\label{equ:Pdn}
P(d, n) := \begin{cases}
2 \!\left\lfloor \tfrac{n}{2d} \right\rfloor & \text{ if $d \equiv 1, 5 \pmod 6$;} \\
2 \!\left\lfloor \tfrac{n}{d} \right\rfloor & \text{ if $d \equiv 2, 4 \pmod 6$;} \\
\left\lfloor \tfrac{3n}{2d} \right\rfloor + \left\lfloor \tfrac{n}{2d} \right\rfloor & \text{ if $d \equiv 3\phantom{,5} \pmod 6$;} \\
\left\lfloor \tfrac{3n}{d} \right\rfloor + \left\lfloor \tfrac{n}{d} \right\rfloor & \text{ if $d \equiv 0\phantom{,5} \pmod 6$;} \\
\end{cases}
\end{equation}
for all integers $d, n \geq 1$.

\begin{lemma}\label{lem:PdLsn}
We have
\begin{equation*}
\mathbb{P}\big[d \in \mathcal{L}_{\mathbf{s}}(n)\big] = 1 - \big(\tfrac1{2}\big)^{P(d, n)} ,
\end{equation*}
for all integers $n \geq 4$ and $d \geq 12$.
\end{lemma}
\begin{proof}
Let $a_1, a_2 \in \{1,2,3,6\}$ and $h_i \in \mathcal{H}_{a_i}(d)$, for $i=1,2$, with $(a_1, h_1) \neq (a_2, h_2)$.
By Lemma~\ref{lem:Had}(iii) we have that $h_1 \neq h_2$.
Moreover, by Lemma~\ref{lem:Had}(i) and $d \geq 12$, we have that
\begin{equation*}
2 \leq \frac{d}{6} \leq \frac{d}{(a_1, a_2)} = \left(\frac{d}{a_1}, \frac{d}{a_2}\right) \mid h_1 - h_2 .
\end{equation*}
Hence, $|h_1 - h_2| \geq 2$ and consequently 
\begin{equation*}
\{2h_1 - 1, 2h_1, 2h_1 + 1\} \cap \{2h_2 - 1, 2h_2, 2h_2 + 1\} = \varnothing .
\end{equation*}
Therefore, the events $\big(h \notin \mathcal{K}_{a,\mathbf{s}}(n)\big)$, with $a \in \{1,2,3,6\}$ and $h \in \mathcal{H}_a(d)$, are mutually independent.
Moreover, we have
\begin{equation*}
\mathbb{P}\big[h \notin \mathcal{K}_{1,\mathbf{s}}(n)\big] = \mathbb{P}\big[h \notin \mathcal{K}_{2,\mathbf{s}}(n)\big] = \tfrac1{4} \quad\text{ and }\quad \mathbb{P}\big[h \notin \mathcal{K}_{3,\mathbf{s}}(n)\big] = \mathbb{P}\big[h \notin \mathcal{K}_{6,\mathbf{s}}(n)\big] = \tfrac1{2} .
\end{equation*}
Thus it follows that
\begin{align*}
\mathbb{P}\big[d \notin \mathcal{L}_{\mathbf{s}}(n)\big] &= \mathbb{P}\bigg[ \bigwedge_{a \,\in\, \{1, 2, 3, 6\}} \bigwedge_{h \,\in\, \mathcal{K}_{a,\mathbf{s}}(n)} \big(d \notin \mathcal{D}_{a}(h)\big)\bigg] \\
 &= \mathbb{P}\bigg[ \bigwedge_{a \,\in\, \{1, 2, 3, 6\}} \bigwedge_{h \,\in\, \mathcal{H}_a(d; n/2)} \big(h \notin \mathcal{K}_{a,\mathbf{s}}(n)\big)\bigg] \\
 &= \prod_{a \,\in\, \{1, 2, 3, 6\}} \prod_{h \,\in\, \mathcal{H}_a(d; n/2)} \mathbb{P}\big[h \notin \mathcal{K}_{a,\mathbf{s}}(n)\big] \\
 &= \big(\tfrac1{4}\big)^{\#\mathcal{H}_1(d, n/2)} \big(\tfrac1{4}\big)^{\#\mathcal{H}_2(d, n/2)} \big(\tfrac1{2}\big)^{\#\mathcal{H}_3(d, n/2)} \big(\tfrac1{2}\big)^{\#\mathcal{H}_6(d, n/2)} = \big(\tfrac1{2}\big)^{P(d, n)} ,
\end{align*}
where the last equality follows by Lemma~\ref{lem:Had}(ii) and~\eqref{equ:Pdn}.
\end{proof}

We are ready to prove Theorem~\ref{thm:random}.
From Lemma~\ref{lem:logellsn_asymp} and Lemma~\ref{lem:PdLsn}, we have that
\begin{align*}\label{equ:logellsn}
\mathbb{E}\!\left[\frac{\log \ell_\mathbf{s}(n)}{\log \alpha}\right] = \sum_{d \,\leq\, 3n} \varphi(d) \,\mathbb{P}\big[d \in \mathcal{L}_\mathbf{s}(n)\big] + O(n) 
= \sum_{d \,\leq\, 3n} \varphi(d) \left(1 - \big(\tfrac1{2}\big)^{P(d, n)}\right) + O(n) ,
\end{align*}
for every sufficiently large integer $n$.
Let $S$ be the last sum.
Splitting $S$ according to the residue class of $d$ modulo $6$, 
and applying Lemma~\ref{lem:phisum}, we get
\begin{align*}
S &= \sum_{d \,\in\, \mathcal{A}_{1,6}(n/2) \,\cup\, \mathcal{A}_{5,6}(n/2)} \varphi(d) \left(1 - \big(\tfrac1{4}\big)^{\left\lfloor \tfrac{n}{2d} \right\rfloor}\right)
   + \sum_{d \,\in\, \mathcal{A}_{2,6}(n) \,\cup\, \mathcal{A}_{4,6}(n)} \varphi(d) \left(1 - \big(\tfrac1{4}\big)^{\left\lfloor \tfrac{n}{d} \right\rfloor}\right) \\
  &\phantom{=} + \sum_{d \,\in\, \mathcal{A}_{3,6}(3n/2)} \varphi(d) \left(1 - \big(\tfrac1{2}\big)^{\left\lfloor \tfrac{3n}{2d} \right\rfloor + \left\lfloor \tfrac{n}{2d} \right\rfloor}\right)
   + \sum_{d \,\in\, \mathcal{A}_{6,6}(3n)} \varphi(d) \left(1 - \big(\tfrac1{2}\big)^{\left\lfloor \tfrac{3n}{d} \right\rfloor + \left\lfloor \tfrac{n}{d} \right\rfloor}\right) \\
  &= \bigg(\tfrac{9}{8}\Li_2\!\big(\tfrac1{4}) + \tfrac{9}{4}\Li_2\!\big(\tfrac1{4}) + \Big(\tfrac{81}{128} + \tfrac{3}{8}\Li_2\!\big(\tfrac1{16}\big) + \tfrac1{16}\Li_2\!\big(\tfrac1{16}; \tfrac1{3}\big) + \tfrac1{32}\Li_2\!\big(\tfrac1{16}; \tfrac{2}{3}\big) \Big) \\
  &\phantom{=} + \Big(\tfrac{81}{64} + \tfrac{3}{4}\Li_2\!\big(\tfrac1{16}\big) + \tfrac1{8}\Li_2\!\big(\tfrac1{16}; \tfrac1{3}\big) + \tfrac1{16}\Li_2\!\big(\tfrac1{16}; \tfrac{2}{3}\big)\Big)\bigg)\frac{n^2}{\pi^2} + O\!\left(n (\log n)^2 \right) \\
  &= C \cdot \frac{n^2}{\pi^2} + O\!\left(n (\log n)^2 \right) .
\end{align*}
The proof is complete.

\bibliographystyle{amsplain}

\end{document}